\documentclass[11pt,a4paper,reqno]{amsart}
\usepackage[english]{babel}
\usepackage[applemac]{inputenc}
\usepackage[T1]{fontenc}
\usepackage{palatino}
\usepackage{cite}
\usepackage{verbatim}
\usepackage{amsmath}
\usepackage{amssymb}
\usepackage{amsthm}
\usepackage{amsfonts}
\usepackage{graphicx}
\usepackage{mathtools}
\usepackage{enumitem}

\usepackage[colorlinks = true, citecolor = black]{hyperref}
\author{Tuomas Orponen}
\title{On the dimension and smoothness of radial projections}
\address{University of Helsinki, Department of Mathematics and Statistics}
\thanks{T.O. is supported by the Academy of Finland via the project \emph{Quantitative rectifiability in Euclidean and non-Euclidean spaces}, grant number 309365. The research was also partially supported by travel grants from the \emph{V\"ais\"al\"a fund} and \emph{Mathematics fund} of the \emph{Finnish Academy of Science and Letters}.}
\email{tuomas.orponen@helsinki.fi}

\newcommand{\R}{\mathbb{R}}

\newcommand{\N}{\mathbb{N}}

\newcommand{\calT}{\mathcal{T}}

\newcommand{\calH}{\mathcal{H}}

\newcommand{\calS}{\mathcal{S}}

\newcommand{\spt}{\operatorname{spt}}
\newcommand{\Hd}{\dim_{\mathrm{H}}}

\newcommand{\1}{\mathbf{1}}

\newcommand{\diam}{\operatorname{diam}}
\newcommand{\card}{\operatorname{card}}
\newcommand{\dist}{\operatorname{dist}}

\newcommand{\dir}{\textup{dir}}

\newcommand{\calM}{\mathcal{M}}
\newcommand{\Inv}{\mathrm{Inv}}

\numberwithin{equation}{section}

\theoremstyle{plain}
\newtheorem{thm}[equation]{Theorem}

\newtheorem{conjecture}[equation]{Conjecture}
\newtheorem{lemma}[equation]{Lemma}

\newtheorem{ex}[equation]{Example}
\newtheorem{cor}[equation]{Corollary}

\theoremstyle{definition}

\theoremstyle{remark}
\newtheorem{remark}[equation]{Remark}

\newcommand{\nref}[1]{(\hyperref[#1]{#1})}

\begin{document}

\begin{abstract} This paper contains two results on the dimension and smoothness of radial projections of sets and measures in Euclidean spaces.

To introduce the first one, assume that $E,K \subset \R^{2}$ are non-empty Borel sets with $\Hd K > 0$. Does the radial projection of $K$ to some point in $E$ have positive dimension? Not necessarily: $E$ can be zero-dimensional, or $E$ and $K$ can lie on a common line. I prove that these are the only obstructions: if $\Hd E > 0$, and $E$ does not lie on a line, then there exists a point in $x \in E$ such that the radial projection $\pi_{x}(K)$ has Hausdorff dimension at least $(\Hd K)/2$. Applying the result with $E = K$ gives the following corollary: if $K \subset \R^{2}$ is Borel set, which does not lie on a line, then the set of directions spanned by $K$ has Hausdorff dimension at least $(\Hd K)/2$.

For the second result, let $d \geq 2$ and $d - 1 < s < d$. Let $\mu$ be a compactly supported Radon measure in $\R^{d}$ with finite $s$-energy. I prove that the radial projections of $\mu$ are absolutely continuous with respect to $\calH^{d - 1}$ for every centre in $\R^{d} \setminus \spt \mu$, outside an exceptional set of dimension at most $2(d - 1) - s$. In fact, for $x$ outside an exceptional set as above, the proof shows that $\pi_{x\sharp}\mu \in L^{p}(S^{d - 1})$ for some $p > 1$. The dimension bound on the exceptional set is sharp.
\end{abstract}

\maketitle

\tableofcontents

\section{Introduction} This paper studies visibility and radial projections. Given $x \in \R^{d}$, define the radial projection $\pi_{x} \colon \R^{d} \setminus \{x\} \to S^{d - 1}$ by
\begin{displaymath} \pi_{x}(y) = \frac{y - x}{|y - x|}. \end{displaymath}
A Borel set $K \subset \R^{2}$ will be called
\begin{itemize}
\item \emph{invisible from $x$}, if $\calH^{d - 1}(\pi_{x}(K \setminus \{x\})) = 0$, and
\item \emph{totally invisible from $x$}, if $\Hd \pi_{x}(K \setminus \{x\}) = 0$. 
\end{itemize}
Above, $\Hd$ and $\calH^{s}$ stand for Hausdorff dimension and $s$-dimensional Hausdorff measure, respectively. I will only consider Hausdorff dimension in this paper, as many of the results below would be much easier for box dimension. The study of (in-)visibility has a long tradition in geometric measure theory. For many more results and questions than I can introduce here, see Section 6 of Mattila's survey \cite{Ma}. The basic question is the following: given a Borel set $K \subset \R^{d}$, how large can the sets
\begin{displaymath} \Inv(K) = \{x \in \R^{d} : K \text{ is invisible from } x\} \end{displaymath}
and
\begin{displaymath} \Inv_{T}(K) := \{x \in \R^{d} : K \text{ is totally invisible from } x\} \end{displaymath}
be? Clearly $\Inv_{T}(K) \subset \Inv(K)$, and one generally expects $\Inv_{T}(K)$ to be significantly smaller than $\Inv(K)$. The existing results fall roughly into the following three categories:
\begin{itemize}
\item[(1)] What happens if $\Hd K > d - 1$? 
\item[(2)] What happens if $\Hd K \leq d - 1$?
\item[(3)] What happens if $0 < \calH^{d - 1}(K) < \infty$?
\end{itemize}

Cases (1) and (3) are the most classical, having already been studied (for $d = 2$) in the 1954 paper \cite{Mar} of Marstrand. Given $s > 1$, Marstrand proved that any Borel set $K \subset \R^{2}$ with $0 < \calH^{s}(K) < 1$ is visible (that is, not invisible) from Lebesgue almost every point $x \in \R^{2}$, and also from $\calH^{s}$ almost every point $x \in K$. Unifying Marstrand's results, and their generalisations to $\R^{d}$, the following sharp bound was recently established by Mattila and the author in \cite{MO} and \cite{O}:
\begin{equation}\label{MOIneq} \Hd \Inv(K) \leq 2(d - 1) - \Hd K, \end{equation}
for all Borel sets $K \subset \R^{d}$ with $d - 1 < \Hd K \leq d$. This paper contains a variant of the bound \eqref{MOIneq} for measures, see Section \ref{intro2}.

The visibility of sets $K$ in Case (3) depends on their rectifiability. I will restrict the discussion to the case $d = 2$ for now. It is easy to show that $1$-rectifiable sets, which are not $\calH^{1}$ almost surely covered by a single line, are visible from all points in $\R^{2}$, with possibly one exception, see \cite{OS}. On the other hand, if $K \subset \R^{2}$ is purely $1$-unrectifiable, then the sharp bound
\begin{displaymath} \Hd [\R^{2} \setminus \Inv(K)] = \Hd \{x \in \R^{2} : K \text{ is visible from } x\} \leq 1. \end{displaymath}
was obtained by Marstrand, building on Besicovitch's projection theorem. For generalisations, improvements and constructions related to the bound above, see \cite[Theorem 5.1]{Ma2}, and \cite{Cs1,Cs2}. Marstrand raised the question -- which remains open to the best of my knowledge -- whether it is possible that $\calH^{1}(\R^{2} \setminus \Inv(K)) > 0$: in particular, can a purely $1$-unrectifiable set be visible from a positive fraction of its own points? For purely $1$-unrectifiable self-similar sets $K \subset \R^{2}$ one has $\Inv(K) = \R^{2}$, as shown by Simon and Solomyak \cite{SS}.

\subsection{The first main result} Case (2) has received less attention. To simplify the discussion, assume that $\Hd K = 1$ and $\calH^{1}(K) = 0$, so that $\Inv(K) = \R^{2}$, and the relevant question becomes the size of $\Inv_{T}(K)$. The radial projections $\pi_{p}$ fit the influential \emph{generalised projections} framework of Peres and Schlag \cite{PS}. If $K \subset \R^{2}$ is a Borel set with arbitrary dimension $s \in [0,2]$, then it follows from \cite[Theorem 7.3]{PS} that
\begin{equation}\label{PSIneq} \Hd \Inv_{T}(K) \leq 2 - s. \end{equation}
When $s > 1$, the bound \eqref{PSIneq} is a weaker version of \eqref{MOIneq}, but the benefit of \eqref{PSIneq} is that it holds without any restrictions on $s$. In particular, if $s = 1$, one obtains
\begin{equation}\label{form41} \Hd \Inv_{T}(K) \leq 1. \end{equation}
This bound is sharp for a trivial reason: consider the case, where $K$ lies on a single line $\ell \subset \R^{2}$. Then, $\Inv_{T}(K) = \ell$. The starting point for this paper was the question: are there essentially different examples manifesting the sharpness of \eqref{form41}? The answer turns out to be negative in a very strong sense. Here are the first main results of the paper:
\begin{thm}[Weak version]\label{mainConceptual} Assume that $K \subset \R^{2}$ is a Borel set with $\Hd K > 0$. Then, at least one of the following holds:
\begin{itemize}
\item $\Hd \Inv_{T}(K) = 0$.
\item $\Inv_{T}(K)$ is contained on a line.
\end{itemize}
\end{thm}
In fact, more is true. For $K \subset \R^{2}$, define
\begin{displaymath} \Inv_{1/2}(K) := \left\{x \in \R^{2} : \Hd \pi_{x}(K \setminus \{x\}) < \tfrac{\Hd K}{2}\right\}. \end{displaymath}
Then, if $\Hd K > 0$, one evidently has $\Inv_{T}(K) \subset \Inv_{1/2}(K) \subset \Inv(K)$.
\begin{thm}[Strong version]\label{main} Theorem \ref{mainConceptual} holds with $\Inv_{T}(K)$ replaced by $\Inv_{1/2}(K)$. That is, if $E \subset \R^{2}$ is a Borel set with $\Hd E > 0$, not contained on a line, then there exists $x \in E$ such that $\Hd \pi_{x}(K \setminus \{x\}) \geq (\Hd K)/2$. \end{thm}
\begin{remark} A closely related result is Theorem 1.6 in the paper \cite{BLZ} of Bond, \L aba and Zahl; with some imagination, Theorem 1.6(a) in \cite{BLZ} can be viewed as a "single scale" variant of Theorem \ref{main}, although at this scale, Theorem 1.6(a) contains more information than Theorem \ref{main}. As far as I can tell, proving the Hausdorff dimension statement in this context presents a substantial extra challenge, so Theorem \ref{main} is not easily implied by the results in \cite{BLZ}.
\end{remark}
\begin{ex} Figure \ref{fig4} depicts the main challenge in the proofs of Theorems \ref{mainConceptual} and \ref{main}. 
\begin{figure}[h!]
\begin{center}
\includegraphics[scale = 0.6]{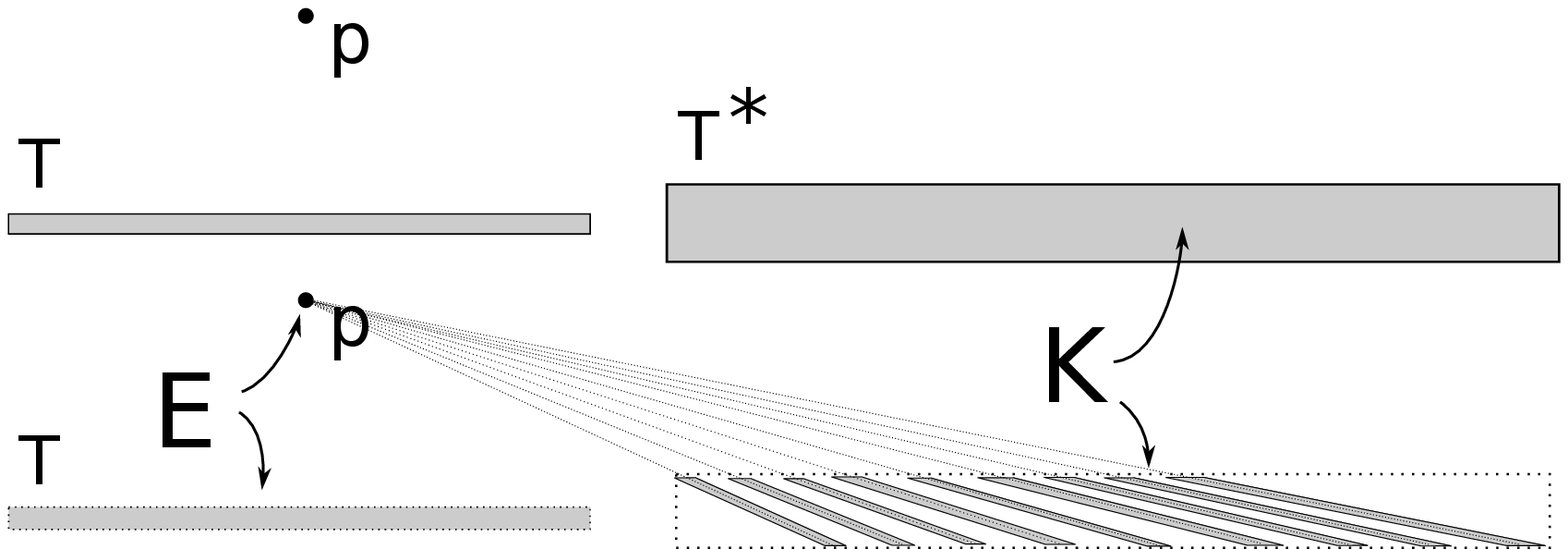}
\caption{What is the next step in the construction of $E$?}\label{fig4}
\end{center}
\end{figure}
The set $E$ has $\Hd E > 0$, and consists of something inside a narrow tube $T$, plus a point $x \notin T$. Then, Theorem \ref{mainConceptual} states that $E \not\subset \Inv_{T}(K)$ for any compact set $K \subset \R^{2}$ with $\Hd K > 0$. So, in order to find a counterexample to Theorem \ref{main}, all one needs to do is find $K$ by a standard "Venetian blind" construction, in such a way that $\Hd K > 0$ and $\Hd \pi_{y}(K) = 0$ for all $y \in E$. The first steps are obvious: to begin with, require that $K \subset T^{\ast}$ for another narrow tube parallel to $T$, see Figure \ref{fig4}. Then $\pi_{y}(K)$ is small for all $y \in T$. To handle the special point $x \in E$, split the contents of $T^{\ast}$ into a finite collection of new narrow tubes in such a way that $\pi_{x}(K)$ is small. In this manner, $\pi_{y}(K)$ can be made arbitrarily small for all $y \in E$ (in the sense of $\epsilon$-dimensional Hausdorff content, for instance, for any prescribed $\epsilon > 0$). It is quite instructive to think, why the construction cannot be completed: why cannot the "Venetian blinds" be iterated further (for both $E$ and $K$) so that, at the limit, $\Hd \pi_{y}(K) = 0$ for all $x \in E$?
\end{ex}

Theorem \ref{main} has the following immediate consequence:
\begin{cor}[Corollary to Theorem \ref{main}]\label{mainCor} Assume that $K \subset \R^{2}$ is a Borel set, not contained on a line. Then the set of unit vectors spanned by $K$, namely
\begin{displaymath} S(K) := \left\{\tfrac{x - y}{|x - y|} \in S^{1} : x,y \in K \text{ and } x \neq y\right\}, \end{displaymath}
satisfies $\Hd S(K) \geq \tfrac{\Hd K}{2}$.
\end{cor}
\begin{proof} If $\Hd K = 0$, there is nothing to prove. Otherwise, Theorem \ref{main} implies that $K \not\subset \Inv_{1/2}(K)$, whence $\Hd S(K) \geq \Hd \pi_{x}(K \setminus \{x\}) \geq (\Hd K)/2$ for some $x \in K$.
\end{proof}

Corollary \ref{mainCor} is probably not sharp, and the following conjecture seems plausible:
\begin{conjecture}\label{mainConj} Assume that $K \subset \R^{2}$ is a Borel set, not contained on a line. Then $\Hd S(K) = \min\{\Hd K,1\}$.
\end{conjecture}
This follows from Marstrand's result, discussed in Case (1) above, when $\Hd K > 1$. For $\Hd K \leq 1$, Conjecture \ref{mainConj} is closely connected with continuous sum-product problems, which means that significant improvements over Corollary \ref{mainCor} will, most likely, require new technology. It would, however, be interesting to know if an $\epsilon$-improvement over Corollary \ref{mainCor} is possible, combining the proof below with ideas from the paper \cite{KT} of Katz and Tao, and using the discretised sum-product theorem of Bourgain \cite{Bo}.

I have the referee to thank for pointing out that a natural discrete variant of Conjecture \ref{mainConj} has been solved by P. Ungar \cite{U} as early as 1982: a set of $n \geq 3$ points in the plane, not all on a single line, determine at least $n - 1$ distinct directions. 

\subsection{The second main result}\label{intro2} The second main result is a version of the estimate \eqref{MOIneq} for measures. Fix $d \geq 2$, and denote the space of compactly supported Radon measures on $\R^{d}$ is denoted by $\calM(\R^{d})$. For $\mu \in \calM(\R^{d})$, write
\begin{displaymath} \calS(\mu) := \{x \in \R^{d} \setminus \spt \mu : \pi_{x\sharp}\mu \text{ is not absolutely continuous w.r.t. } \calH^{d - 1}|_{S^{d - 1}}\}. \end{displaymath}
Note that whenever $x \in \R^{d} \setminus \spt \mu$, the projection $\pi_{x}$ is continuous on $\spt \mu$, and $\pi_{x\sharp}\mu$ is well-defined. One can check that the family of projections $\{\pi_{x}\}_{x \in \R^{d} \setminus \spt \mu}$ fits in the \emph{generalised projections} framework of Peres and Schlag \cite{PS}, and indeed Theorem 7.3 in \cite{PS} yields
\begin{equation}\label{PSIneqMeasures} \Hd \calS(\mu) \leq 2d - 1 - s, \end{equation}
whenever $d - 1 < s < d$ and $\mu \in \calM(\R^{d})$ has finite $s$-energy (see \eqref{sEnergy} for a definition). Combining this bound with standard arguments shows that if $K \subset \R^{d}$ is a Borel set with $d - 1 < \Hd K \leq d$, then 
\begin{displaymath} \Hd \Inv(K) = \Hd \{x \in \R^{d} : \calH^{d - 1}(\pi_{x}(K)) = 0\} \leq 2d - 1 - \Hd K. \end{displaymath}
This is weaker than the sharp bound \eqref{MOIneq}, so it is a natural to ask, whether the bound \eqref{PSIneqMeasures} for measures could be lowered to match \eqref{MOIneq}. The answer is affirmative:

\begin{thm}\label{mainRadial} If $\mu \in \calM(\R^{d})$ and 
\begin{equation}\label{sEnergy} I_{s}(\mu) := \iint \frac{d\mu(x) \, d\mu(y)}{|x - y|^{s}} < \infty \end{equation}
for some $s > d - 1$, then $\Hd \calS(\mu) \leq 2(d - 1) - s$.
\end{thm}
The bound is sharp, essentially because \eqref{MOIneq} is, and Theorem \ref{mainRadial} implies \eqref{MOIneq}. More precisely, following \cite[Section 2.2]{O}, there exist compact sets $K \subset \R^{d}$ of any dimension $\Hd K \in (d - 1,d)$ such that 
\begin{displaymath} \Hd [\Inv(K) \setminus K] = 2(d - 1) - \dim K. \end{displaymath}
Then, the sharpness of Theorem \ref{mainRadial} follows by considering Frostman measures supported on $K$, and noting that $\mathcal{S}(\mu) \supset \Inv(K) \setminus K$ whenever $\mu \in \mathcal{M}(\R^{d})$ and $\spt \mu \subset K$. 

An open question is the validity of Theorem \ref{mainRadial} for $s = d - 1$. If $I_{d - 1}(\mu) < \infty$, Theorem 7.3 in \cite{PS} implies that $\mathcal{L}^{d}(\mathcal{S}(\mu)) = 0$, but I do not even know if $\Hd \mathcal{S}(\mu) < d$. 

Theorem \ref{mainRadial} does not immediately follow from the proof of \eqref{MOIneq} in \cite{MO} and \cite{O}, as the argument in those papers was somewhat indirect. Having said that, many observations from the previous papers still play a role in the new proof. Theorem \ref{mainRadial} will be deduced from the next statement concerning $L^{p}$-densities:

\begin{thm}\label{mainTechnical} Let $\mu \in \calM(\R^{d})$ as in Theorem \ref{main}. For $p \in (1,2)$, write 
\begin{displaymath} \calS_{p}(\mu) := \{x \in \R^{d} \setminus \spt \mu : \pi_{x\sharp}\mu \notin L^{p}(S^{d - 1})\}. \end{displaymath}
Then $\Hd \calS_{p}(\mu) \leq 2(d - 1) - s + \delta(p)$, where $\delta(p) > 0$, and $\delta(p) \to 0$ as $p \searrow 1$.
 \end{thm}
 
Note that the claim is vacuous for "large" values of $p$. The dependence of $\delta(p) > 0$ on $p$ is effective and not very hard to track, see \eqref{assumptions}.
 
\begin{remark} Theorem \ref{mainTechnical} can be viewed as an extension of Falconer's exceptional set estimate \cite{Fa} from 1982. I only discuss the planar case. Falconer proved that if $I_{s}(\mu) < \infty$ for some $1 < s < 2$, then the orthogonal projections of $\mu$ to all $1$-dimensional subspaces are in $L^{2}$, outside an exceptional set of dimension at most $2 - s$. Now, orthogonal projections can be viewed as radial projections from points on the line at infinity. Alternatively, if the reader prefers a more rigorous statement, Falconer's proof shows that if $\ell \subset \R^{2}$ is any fixed line outside the support of $\mu$, then all the radial projections of $\mu$ to points on $\ell$ are in $L^{2}$, outside an exceptional set of dimension at most $2 - s$. In comparison, Theorem \ref{mainTechnical} states that the radial projections of $\mu$ to points in $\R^{2} \setminus \spt \mu$ are in $L^{p}$ for some $p > 1$, outside an exceptional set of dimension at most $2 - s$. So, the size of the exceptional set remains the same even if the "fixed line $\ell$" is removed from the statement. The price to pay is that the projections only belong to some $L^{p}$ with $p > 1$ (possibly) smaller than $2$. I do not know, if the reduction in $p$ is necessary, or an artefact of the proof. \end{remark} 
 
\subsection{Acknowledgements} I started working on the questions while taking part in the research programme \emph{Fractal Geometry and Dynamics} at Institut Mittag-Leffler. I am grateful to the organisers for letting me participate, and to the staff of the institute for making my stay very pleasant. I would also like to thank Tam\'as Keleti and Pablo Shmerkin for stimulating conversations, both on this project, and several related topics. I thank Pablo for explicitly asking, whether Theorem \ref{mainRadial} is true. The visit at the institute was enabled financially by travel grants from the \emph{V\"ais\"al\"a fund} and \emph{Mathematics fund} of the \emph{Finnish Academy of Science and Letters}. 

I am grateful to the referee for a careful reading of the manuscript, and for many helpful suggestions.

\section{Proof of Theorem \ref{main}}

If $\ell \subset \R^{2}$ is a line, I denote by $T(\ell,\delta)$ the open (infinite) tube of width $2\delta$, with $\ell$ "running through the middle", that is, $\dist(\ell,\R^{2} \setminus T(\ell,\delta)) = \delta$. The notation $B(x,r)$ stands for a closed ball with centre $x \in \R^{2}$ and radius $r > 0$. The notation $A \lesssim B$ means that there is an absolute constant $C \geq 1$ such that $A \leq CB$.

\begin{lemma}\label{auxLemma1} Assume that $\mu$ is a Borel probability measure on $B(0,1) \subset \R^{2}$, and $\mu(\ell) = 0$ for all lines $\ell \subset \R^{2}$. Then, for any $\epsilon > 0$, there exists $\delta > 0$ such that $\mu(T(\ell,\delta)) \leq \epsilon$ for all lines $\ell \subset \R^{2}$.
\end{lemma}

\begin{proof} Assume not, so there exists $\epsilon > 0$, a sequence of positive numbers $\delta_{1} > \delta_{2} > \ldots > 0$ with $\delta_{i} \searrow 0$, and a sequence of lines $\{\ell_{i}\}_{i \in \N} \subset \R^{2}$ with $\mu(T(\ell_{i},\delta_{i})) \geq \epsilon$. Since $\spt \mu \subset B(0,1)$, one has $\ell_{i} \cap B(0,1) \neq \emptyset$ for all $i \in \N$. Consequently, there exists a subsequence $(i_{j})_{j \in \N}$, and a line $\ell \subset \R^{2}$ such that $\ell_{j} \to \ell$ in the Hausdorff metric. Then, for any given $\delta > 0$, there exists $j \in \N$ such that
\begin{displaymath} B(0,1) \cap T(\ell_{i_{j}},\delta_{i_{j}}) \subset T(\ell,\delta), \end{displaymath}
so that $\mu(T(\ell,\delta)) \geq \epsilon$. It follows that $\mu(\ell) \geq \epsilon$, a contradiction. \end{proof}

The next lemma contains all the information needed to prove Theorem \ref{main}. I state two versions: the first one is slightly easier to read and apply, while the second one is slightly more detailed. 
\begin{lemma}\label{mainLemmaSimplified} Assume that $\mu,\nu$ are Borel probability measures with compact supports $K,E \subset B(0,1)$, respectively. Assume that both measures $\mu$ and $\nu$ satisfy a Frostman condition with exponents $\kappa_{\mu},\kappa_{\nu} \in (0,2]$, respectively:
\begin{equation}\label{frostmanBound} \mu(B(x,r)) \leq C_{\mu} r^{\kappa_{\mu}} \quad \text{and} \quad \nu(B(x,r)) \leq C_{\nu} r^{\kappa_{\nu}} \end{equation}
for all balls $B(x,r) \subset \R^{2}$, and for some constants $C_{\mu},C_{\nu} \geq 1$. Assume further that $\mu(\ell) = 0$ for all lines $\ell \subset \R^{2}$. Fix also 
\begin{displaymath} 0 < \tau < \tfrac{\kappa_{\mu}}{2} \quad \text{and} \quad \epsilon > 0, \end{displaymath}
and write $\delta_{k} := 2^{-(1 + \epsilon)^{k}}$. 

Then, there exists a compact subset $K' \subset K$ with
\begin{displaymath} \mu(K') \geq \frac{1}{2}, \end{displaymath}
a number $\eta = \eta(\epsilon,\kappa_{\mu},\kappa_{\nu},\tau) > 0$, an index $k_{0} = k_{0}(\epsilon,\mu,\kappa_{\nu},\tau) \in \N$, and a point $x \in E$ with the following property. If $k > k_{0}$, and $T(\ell_{1},\delta_{k}),\ldots,T(\ell_{N},\delta_{k})$ is a family of $\delta_{k}$-tubes of cardinality $N \leq \delta_{k}^{-\tau}$, each containing $x$, then
\begin{equation}\label{form1b} \mu\left(K' \cap \bigcup_{j = 1}^{N} T(\ell_{j},\delta_{k})\right) \leq \delta_{k}^{\eta}. \end{equation}
\end{lemma}
Roughly speaking, the conclusion \eqref{form1b} means that $K'$ has a radial projection of dimension $\geq \tau$ relative to the viewpoint $x \in E$, since only a tiny fraction of $K'$ can be covered by $\leq \delta_{k}^{-\tau}$ tubes of width $2\delta_{k}$ containing $x$. 

The set $K' \subset K$ and the point $x \in E$ will be found by induction on the scales $\delta_{k}$. To set the scene for the induction, it is convenient to state a more detailed version of the lemma:

\begin{lemma}\label{mainLemma} Assume that $\mu,\nu$ are Borel probability measures with compact supports $K,E \subset B(0,1)$, respectively. Assume that both measures $\mu$ and $\nu$ satisfy a Frostman condition with exponents $\kappa_{\mu},\kappa_{\nu} \in (0,2]$, respectively:
\begin{equation}\label{frostmanBound} \mu(B(x,r)) \leq C_{\mu} r^{\kappa_{\mu}} \quad \text{and} \quad \nu(B(x,r)) \leq C_{\nu} r^{\kappa_{\nu}} \end{equation}
for all balls $B(x,r) \subset \R^{2}$, and for some constants $C_{\mu},C_{\nu} \geq 1$. Assume further that $\mu(\ell) = 0$ for all lines $\ell \subset \R^{2}$. Fix also
\begin{displaymath} 0 < \tau < \tfrac{\kappa_{\mu}}{2} \quad \text{and} \quad \epsilon > 0, \end{displaymath}
and write $\delta_{k} := 2^{-(1 + \epsilon)^{k}}$.

Then, there exist numbers $\beta = \beta(\kappa_{\mu},\kappa_{\nu},\tau) > 0$, $\eta = \eta(\epsilon,\kappa_{\mu},\kappa_{\nu},\tau) > 0$, and an index $k_{0} = k_{0}(\epsilon,\mu,\kappa_{\nu},\tau) \in \N$ with the following properties. For all $k \geq k_{0}$, there exist
\begin{itemize}
\item[(a)] compact sets $K \supset K_{k_{0}} \supset K_{k_{0} + 1} \ldots$ with 
\begin{equation}\label{form6} \mu(K_{k}) \geq 1 - \sum_{k_{0} \leq j < k} (\tfrac{1}{4})^{j - k_{0} + 1} \geq \frac{1}{2}, \end{equation}
\item[(b)] compact sets $E \supset E_{k_{0}} \supset E_{k_{0} + 1} \ldots$ with $\nu(E_{k}) \geq \delta_{k}^{\beta}$
\end{itemize}
with the following property: if $k > k_{0}$, $x \in E_{k}$, and $T(\ell_{1},\delta_{k}),\ldots,T(\ell_{N},\delta_{k})$ is a family of tubes of cardinality $N \leq \delta_{k}^{-\tau}$, each containing $x$, then
\begin{equation}\label{form1a} \mu\left(K_{k} \cap \bigcup_{j = 1}^{N} T(\ell_{j},\delta_{k})\right) \leq \delta_{k}^{\eta}. \end{equation}
\end{lemma}

\begin{remark}\label{k0Large} The index $k_{0}$ can be chosen as large as desired; this will be clear from the proof below. It will also be used on many occasions, without separate remark, that $\delta_{k}$ can be assumed very small for all $k \geq k_{0}$. I also record that Lemma \ref{mainLemmaSimplified} follows from Lemma \ref{mainLemma}: simply take $K'$ to be the intersection of all the sets $K_{j}$, $j \geq k_{0}$, and let $x \in E$ be any point in the intersection of all the sets $E_{j}$, $j \geq k_{0}$. \end{remark}

\begin{proof} As stated above, the proof is by induction, starting at the largest scale $k_{0}$, which will be presently defined. Fix $\eta = \eta(\epsilon,\kappa_{\mu},\kappa_{\nu},\tau) > 0$ and
\begin{equation}\label{form27} \Gamma = \Gamma(\epsilon,\kappa_{\mu},\kappa_{\nu},\tau) \in \N \end{equation}
The number $\Gamma$ will be specified at the very end of the proof, right before \eqref{form32}, and there will be several requirements for the number $\eta$, see \eqref{form33}, \eqref{form25}, and \eqref{form34}. Applying Lemma \ref{auxLemma1}, first pick an index $k_{1} = k_{1}(\epsilon,\mu,\kappa_{\nu},\tau) \in \N$ such that $\mu(T(\ell,\delta_{k_{1}})) \leq (\tfrac{1}{4})^{\Gamma + 1}$ for all tubes $T(\ell,\delta_{k_{1}}) \subset \R^{2}$, and
\begin{equation}\label{form13} \delta_{k - \Gamma}^{\eta} \leq (\tfrac{1}{4})^{k - \Gamma + 1}, \qquad k \geq k_{1}. \end{equation}
Set $k_{0} := k_{1} + \Gamma$. Then, the following holds for all $k \in \{k_{0},\ldots,k_{0} + \Gamma\}$. For any subset $K' \subset K$, and any tube $T(\ell,\delta_{k - \Gamma}) \subset \R^{2}$, one has 
\begin{equation}\label{form12} \mu(K' \cap T(\ell,\delta_{k - \Gamma})) \leq \mu(T(\ell,\delta_{k_{1}})) \leq (\tfrac{1}{4})^{\Gamma + 1} \leq (\tfrac{1}{4})^{k - k_{0} + 1}. \end{equation} 
Define
\begin{displaymath} K_{k} := K \quad \text{and} \quad E_{k} := E, \qquad k_{1} \leq k \leq k_{0}. \end{displaymath}
(The definitions of $E_{k},K_{k}$ for $k_{1} \leq k < k_{0}$ are only given for notational convenience.)

I start by giving an outline of how the induction will proceed. Assume that, for a certain $k \geq k_{0}$, the sets $K_{k}$ and $E_{k}$ have been constructed such that
\begin{itemize}
\item[(i)] the condition \eqref{form12} is satisfied with $K' = K_{k}$, and for all tubes $T(\ell,\delta_{k - \Gamma})$ with $T(\ell,\delta_{k - \Gamma}) \cap E_{k - \Gamma} \neq \emptyset$.
\item[(ii)] $K_{k}$ and $E_{k}$ satisfy the measure lower bounds (a) and (b) from the statement of the lemma.
\end{itemize}
Under the conditions (i)-(ii), I claim that it is possible to find subsets $K_{k + 1} \subset K_{k}$ and $E_{k + 1} \subset E_{k}$, satisfying (ii) at level $k + 1$, and also the non-concentration condition \eqref{form1a} at level $k + 1$. This is why \eqref{form1a} is only claimed to hold for $k > k_{0}$, and no one is indeed claiming that it holds for the sets $K_{k_{0}}$ and $E_{k_{0}}$. These sets satisfy (i), however, which should be viewed as a weaker substitute for \eqref{form1a} at level $k$, which is just strong enough to guarantee \eqref{form1a} at level $k + 1$. There is one obvious question at this point: if (i) at level $k$ gives \eqref{form1a} at level $k + 1$, then where does one get (i) back at level $k + 1$?
 
If $k + 1 \in \{k_{0},\ldots,k_{0} + \Gamma\}$, the condition (i) is simply guaranteed by the choice of $k_{0}$ (one does not even need to assume that $T(\ell,\delta_{k - \Gamma}) \cap E_{k - \Gamma} \neq \emptyset$). For $k + 1 > k_{0} + \Gamma$, this is no longer true. However, for $k + 1 > \Gamma + k_{0}$, one has $k + 1 - \Gamma > k_{0}$, and thus $K_{k + 1- \Gamma}$ and $E_{k + 1- \Gamma}$ have already been constructed to satisfy \eqref{form1a}. In particular, if $E_{k + 1 - \Gamma} \cap T(\ell,\delta_{k + 1 - \Gamma}) \neq \emptyset$, then
\begin{equation}\label{form26} \mu(K_{k + 1} \cap T(\ell,\delta_{k + 1 - \Gamma})) \leq \mu(K_{k + 1 - \Gamma} \cap T(\ell,\delta_{k + 1- \Gamma})) \leq \delta_{k + 1- \Gamma}^{\eta} \leq (\tfrac{1}{4})^{(k + 1) - k_{0} + 1} \end{equation}
 by \eqref{form1a} and \eqref{form13}. This means that (i) is satisfied at level $k + 1$, and the induction may proceed.  
 
So, it remains to prove that (i)--(ii) at level $k$ imply (ii) and \eqref{form1a} at level $k + 1$. To avoid clutter, I write
\begin{displaymath} \delta := \delta_{k + 1}. \end{displaymath}
Assume that the sets $K_{k},E_{k}$ have been constructed for some $k \geq k_{0}$, satisfying (i)--(ii). The main task is to understand the structure of the set of points $x \in E_{k}$ for which \eqref{form1a} fails. To this end, we define the set $\textbf{Bad}_{k} \subset E_{k}$ as follows: $x \in \textbf{Bad}_{k}$, if and only if $x \in E_{k}$, and there exist $N \leq \delta^{-\tau}$ tubes $T(\ell_{1},\delta),\ldots,T(\ell_{N},\delta)$, each containing $x$, such that
\begin{equation}\label{form23} \mu\left(K_{k} \cap \bigcup_{j = 1}^{N} T(\ell_{j},\delta) \right) > \delta^{\eta}. \end{equation}
Note that if $\textbf{Bad}_{k} = \emptyset$, then one can simply define $E_{k + 1} := E_{k}$ and $K_{k + 1} := K_{k}$, and (ii) and \eqref{form1a} (at level $k + 1$) are clearly satisfied.

Instead of analysing $\textbf{Bad}_{k}$ directly, it is useful to split it up into "directed" pieces, and digest the pieces individually. To make this precise, let $S$ be the "space of directions"; for concreteness, I identify $S$ with the upper half of the unit circle. Then, if $T = T(\ell,\delta) \subset \R^{2}$ is a tube, I denote by $\dir(T)$ the unique vector $e \in S$ such that $\ell \| e$.

Recall the small parameter $\eta > 0$, and partition $S$ into $D = \delta^{-\eta}$ arcs $J_{1},\ldots,J_{D}$ of length $\sim \delta^{\eta}$.\footnote{Here, it might be better style to pick another letter, say $\alpha > 0$, in place of $\eta$, since the two parameters play slightly different roles in the proof. Eventually, however, one would end up considering $\min\{\eta,\alpha\}$, and it seems a bit cleaner to let $\eta > 0$ be a "jack of all trades" from the start.} For $d \in \{1,\ldots,D\}$ fixed ("$d$" for "direction"), consider the set $\textbf{Bad}_{k}^{d}$: it consists of those points $x \in E_{k}$ such that there exist $N \leq \delta^{-\tau}$ tubes $T(\ell_{1},\delta),\ldots,T(\ell_{N},\delta)$, each containing $x$, with $\dir(T(\ell_{i},\delta)) \in J_{d}$, and satisfying
\begin{displaymath} \mu\left(K_{k} \cap \bigcup_{j = 1}^{N} T(\ell_{j},\delta) \right) > \delta^{2\eta}. \end{displaymath}
Since the direction of every possible tube in $\R^{2}$ belongs to one of the arcs $J_{i}$, and there are only $D = \delta^{-\eta}$ arcs in total, one has
\begin{equation}\label{form14} \textbf{Bad}_{k} \subset \bigcup_{d = 1}^{D} \textbf{Bad}_{k}^{d}. \end{equation}

The next task is to understand the structure of $\textbf{Bad}_{k}^{d}$ for a fixed direction $d \in \{1,\ldots,D\}$. I claim that $\textbf{Bad}_{k}^{d}$ looks like a garden of flowers, with all the petals pointing in direction $J_{d}$, see Figure \ref{fig2} for a rough idea.
\begin{figure}[h!]
\begin{center}
\includegraphics[scale = 0.6]{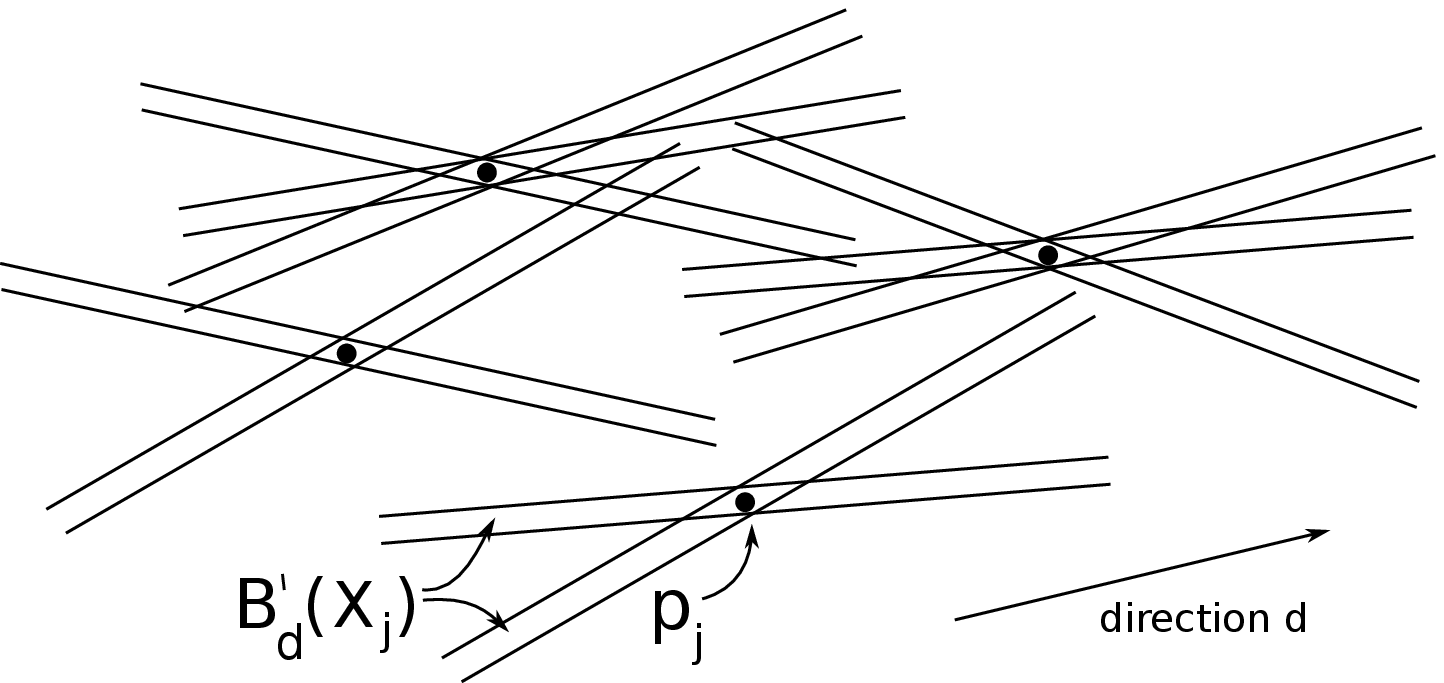}
\caption{The set $\textbf{Bad}_{k}^{d}$.}\label{fig2}
\end{center}
\end{figure}
To make the statement more precise, I introduce an additional piece of notation. Fox $X \subset K_{k}$, let $B_{d}(X)$ consist of those points $x \in E_{k}$ such that $X$ can be covered by $N \leq \delta^{-\tau}$ tubes $T(\ell_{1},\delta),\ldots,T(\ell_{N},\delta)$, with directions $\dir(T(\ell_{i},\delta)) \in J_{d}$, and each containing $x$. Then, note that
\begin{equation}\label{form35} \textbf{Bad}_{k}^{d} = \{x \in E_{k} : \exists \, X \subset K_{k} \text{ with } \mu(X) > \delta^{2\eta} \text{ and } x \in B_{d}(X)\}. \end{equation}
The sets $B_{d}(X)$ also have the trivial but useful property that 
\begin{displaymath} X \subset X' \subset K_{k} \quad \Longrightarrow \quad B_{d}(X') \subset B_{d}(X). \end{displaymath}

There are two steps in establishing the "garden" structure of $\textbf{Bad}_{k}^{d}$: first, one needs to find the "flowers", and second, one needs to check that the sets obtained actually look like flowers in a non-trivial sense. I start with the former task. Assuming that $\textbf{Bad}_{k}^{d} \neq \emptyset$, pick any point $x_{1} \in \textbf{Bad}_{k}^{d}$, and an associated subset $X_{1} \subset K_{k}$ with 
\begin{displaymath} \mu(X_{1}) > \delta^{2\eta} \quad \text{and} \quad x_{1} \in B_{d}(X_{1}). \end{displaymath}
Then, assume that $x_{1},\ldots,x_{m} \in \textbf{Bad}_{k}^{d}$ and $X_{1},\ldots,X_{m}$ have already been chosen with the properties above, and further satisfying
\begin{equation}\label{form15} \mu(X_{i} \cap X_{j}) \leq \delta^{4\eta}/2, \qquad 1 \leq i < j \leq m. \end{equation}
Then, see if there still exists a subset $X_{m + 1} \subset K_{k}$ with the following three properties: $\mu(X_{m + 1}) > \delta^{2\eta}$, $B_{d}(X_{m + 1}) \neq \emptyset$, and $\mu(X_{m + 1} \cap X_{i}) \leq \delta^{4\eta}/2$ for all $1 \leq i \leq m$. If such a set no longer exists, stop; if it does, pick $x_{m + 1} \in B_{d}(X_{m + 1})$, and add $X_{m + 1}$ to the list.

It follows from the "competing" conditions $\mu(X_{i}) > \delta^{2\eta}$, and \eqref{form15}, that the algorithm needs to terminate in at most
\begin{equation}\label{form16} M \leq 2\delta^{-4\eta} \end{equation}
Indeed, assume that the sets $X_{1},\ldots,X_{M}$ have already been constructed, and consider the following chain of inequalities:

\begin{align*} \frac{1}{M} + \frac{1}{M(M - 1)}\sum_{i_{1} \neq i_{2}} \mu(X_{i_{1}} \cap X_{i_{2}}) & \geq \frac{1}{M^{2}}\sum_{i_{1},i_{2} = 1}^{M} \mu(X_{i_{1}} \cap X_{i_{2}})\\
& = \frac{1}{M^{2}} \int \sum_{i_{1},i_{2} = 1}^{M} \1_{X_{i_{1}} \cap X_{i_{2}}}(x) \, d\mu(x)\\
& = \frac{1}{M^{2}} \int [\card \{1 \leq i \leq M : x \in X_{i}\}]^{2} \, d\mu(x)\\
& \geq \frac{1}{M^{2}} \left( \int \card\{1 \leq i \leq M : x \in X_{i}\} \, d\mu(x) \right)^{2}\\
& = \frac{1}{M^{2}} \left( \sum_{i = 1}^{M} \mu(X_{i}) \right)^{2} > \delta^{4\eta}. \end{align*}
Thus, if $M > 2\delta^{-4\eta}$, there exists a pair $X_{i_{1}},X_{i_{2}}$ with $i_{1} \neq i_{2}$ such that $\mu(X_{i_{1}} \cap X_{i_{2}}) > \delta^{4\eta}/2$, and the algorithm has already terminated earlier. This proves \eqref{form16}.

With the sets $X_{1},\ldots,X_{M}$ now defined, write
\begin{displaymath} B_{d}'(X_{j}) := \{x \in E_{k} : \exists \, X' \subset X_{j} \text{ with } \mu(X') > \delta^{4\eta}/2 \text{ and } p \in B_{d}(X')\}.  \end{displaymath}
I claim that
\begin{equation}\label{form17} \textbf{Bad}_{k}^{d} \subset \bigcup_{j = 1}^{M} B_{d}'(X_{j}). \end{equation}
Indeed, if $x \in \textbf{Bad}_{k}^{d}$, then $x \in B_{d}(X)$ for some $X \subset K_{k}$ with $\mu(X) > \delta^{2\eta}$ by \eqref{form35}. It follows that 
\begin{equation}\label{form18} \mu(X \cap X_{j}) > \delta^{4\eta}/2 \end{equation}
for one of the sets $X_{j}$, $1 \leq j \leq M$, because either $X \in \{X_{1},\ldots,X_{M}\}$, and \eqref{form18} is clear (all the sets $X_{j}$ even satisfy $\mu(X_{j}) > \delta^{2\eta}$), or else \eqref{form18} must hold by virtue of $X$ \textbf{not} having been added to the list $X_{1},\ldots,X_{M}$ in the algorithm. But \eqref{form18} implies that $x \in B_{d}'(X_{j})$, since $X' = X \cap X_{j} \subset X_{j}$ satisfies $\mu(X') > \delta^{4\eta}/2$ and $x \in B_{d}(X) \subset B_{d}(X')$.

According to \eqref{form16} and \eqref{form17} the set $\textbf{Bad}_{k}^{d}$ can be covered by $M \leq 2\delta^{-4\eta}$ sets of the form $B_{d}'(X_{j})$, see Figure \ref{fig2}. These sets are the "flowers", and their structure is explored in the next lemma:
\begin{lemma}\label{auxLemma3} The following holds, if $\delta = \delta_{k + 1}$ and $\eta > 0$ are small enough (the latter depending on $\kappa_{\mu},\tau$ here). For $1 \leq d \leq D$ and $1 \leq j \leq M$ fixed, the set $B_{d}'(X_{j})$ can be covered by $\leq 4\delta^{-8\eta}$ tubes of the form $T = T(\ell,\delta^{\rho})$, where $\dir(T) \in J_{d}$, and $\rho = \rho(\kappa_{\mu},\tau) > 0$. The tubes can be chosen to contain the point $x_{j} \in B_{d}(X_{j})$.
\end{lemma}

\begin{proof} Fix $1 \leq j \leq M$ and $x \in B_{d}'(X_{j})$. Recall the point $x_{j} \in B_{d}(X_{j})$ from the definition of $X_{j}$. By definition of $x \in B_{d}'(X_{j})$, there exists a set $X' \subset X_{j}$ with $\mu(X') > \delta^{4\eta}/2$ and $x \in B_{d}(X')$. Unwrapping the definitions further, there exist $N \leq \delta^{-\tau}$ tubes $T(\ell_{1},\delta),\ldots,T(\ell_{N},\delta)$, the union of which covers $X'$, and each satisfies $\dir(T(\ell_{i},\delta)) \in J_{d}$ and $x \in T(\ell_{i},\delta)$. In particular, one of these tubes, say $T_{x} = T(\ell_{i},\delta)$, has
\begin{equation}\label{form19} \mu(X_{j} \cap T_{x}) \geq \mu(X' \cap T_{x}) \geq \mu(X') \cdot \delta^{\tau} \geq \delta^{4\eta + \tau}/2 \geq \delta^{8\eta + \tau}/4. \end{equation}
(The final inequality is just a triviality at this point, but is useful for later technical purposes later.) Here comes perhaps the most basic geometric observation in the proof: if the measure lower bound \eqref{form19} holds for some $\delta$-tube $T$ -- this time $T_{x}$ --  and a sufficiently small $\eta > 0$ (crucially so small that $8\eta + \tau < \kappa_{\mu}/2$), then the whole set $B_{d}(X_{j})$ is actually contained in a neighbourhood of $T$, called $T^{\ast}$, because $X_{j} \cap T$ is so difficult to cover by $\delta$-tubes centred at points outside $T^{\ast}$, see Figure \ref{fig3}. \begin{figure}[h!]
\begin{center}
\includegraphics[scale = 0.4]{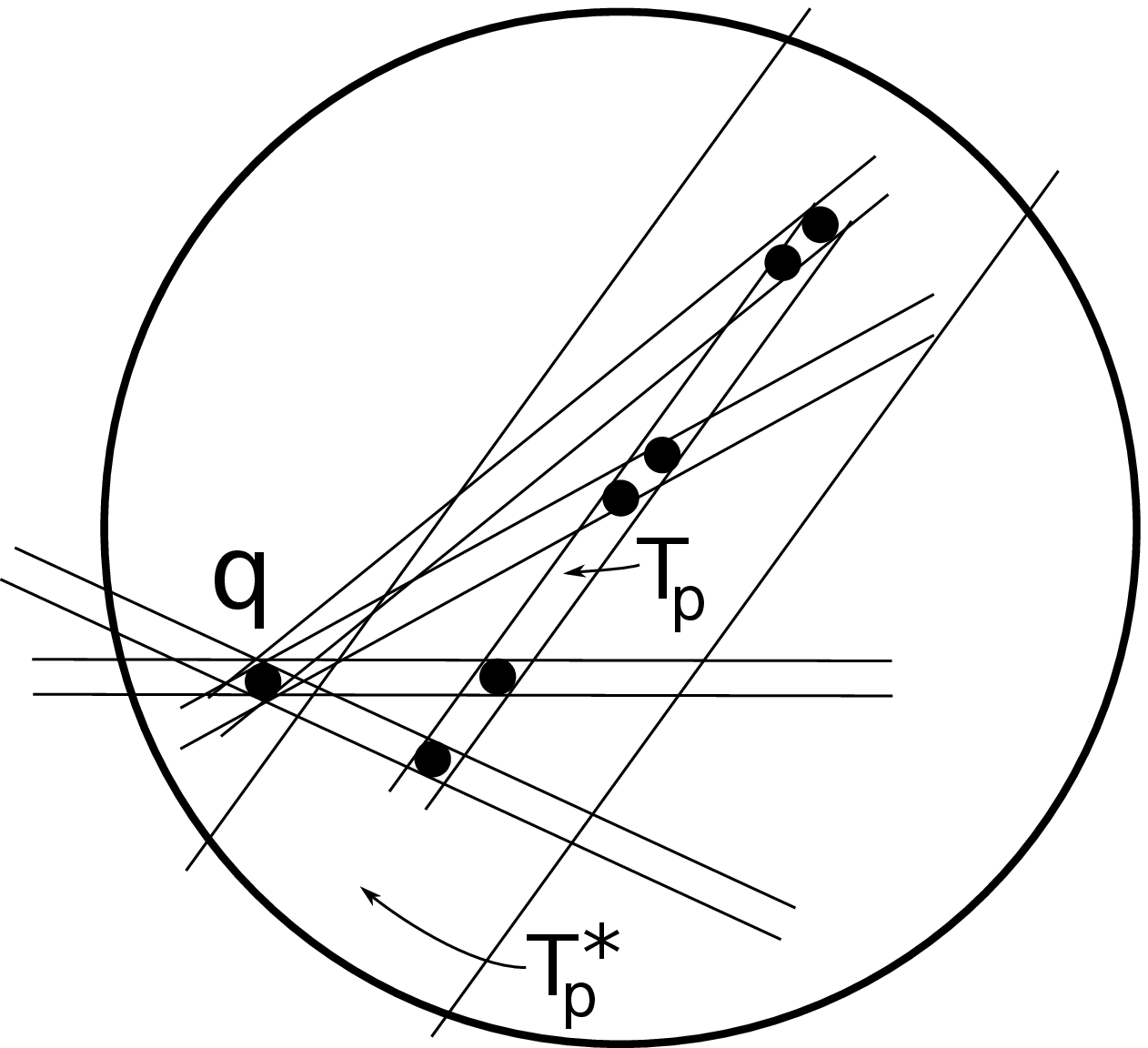}
\caption{Covering $X_{j} \cap T_{x}$ by tubes centred at points outside $T_{x}^{\ast}$.}\label{fig3}
\end{center}
\end{figure}
In particular, in the present case,
\begin{equation}\label{form36} x_{j} \in B_{d}(X_{j}) \subset T(\ell_{i},\delta^{4\rho}) =: T_{x}^{\ast} \end{equation}
for a suitable constant $\rho = \rho(\kappa_{\mu},\tau) > 0$, specified in \eqref{form33}. To see this formally, pick $y \in B(0,1) \setminus T_{x}^{\ast}$, and argue as follows to show that $y \notin B_{d}(X_{j})$. First, any $\delta$-tube $T$ containing $y$, and intersecting $T_{x} \cap B(0,1)$, makes an angle of at least $\gtrsim \delta^{4\rho}$ with $T_{x}$. It follows that
\begin{displaymath} \diam(T \cap T_{x} \cap B(0,1)) \lesssim \delta^{1 - 4\rho}, \end{displaymath}
and consequently $\mu(T \cap T_{x} \cap B(0,1)) \lesssim C_{\mu}\delta^{\kappa_{\mu}(1 - 4\rho)}$. So, in order to cover $X_{j} \cap T_{x}$ (let alone the whole set $X_{j}$) it takes by \eqref{form19} at least
\begin{equation}\label{form20} \gtrsim \frac{\mu(X_{j} \cap T_{x})}{C_{\mu}\delta^{\kappa_{\mu}(1 - 4\rho)}} \geq \frac{\delta^{8\eta + \tau - \kappa_{\mu}(1 - 4\rho)}}{4C_{\mu}} \geq \frac{\delta^{8\eta - \kappa_{\mu}/2 + 8\rho}}{4C_{\mu}} \end{equation}
tubes $T$ containing $y$. But if 
\begin{equation}\label{form33} 0 < 8\eta < \frac{\tfrac{\kappa_{\mu}}{2} - \tau}{2} \quad \text{and} \quad 8\rho = \frac{\tfrac{\kappa_{\mu}}{2} - \tau}{2}, \end{equation}
then the number on the right hand side of \eqref{form20} is far larger than $\delta^{-\tau}$, which means that $y \notin B_{d}(X_{j})$, and proves \eqref{form36}. 

Recall the statement of the Lemma \ref{auxLemma3}, and compare it with the previous accomplishment: \eqref{form36} states that whenever $x \in B'_{d}(X_{j})$, then $x$ lies in a certain tube of width $\delta^{4\rho}$ (namely $T_{x}$), which has direction in $J_{d}$, and also contains $x_{j}$. This sounds a bit like the statement of the lemma, but there is a problem: in principle, every point $x \in B'(X_{j})$ could give rise to a different tube $T_{x}$. So, it essentially remains to show that all these $\delta^{4\rho}$-tubes $T_{x}$ can be covered by a small number of tubes of width $\delta^{\rho}$. To begin with, note that the ball $B_{j} := B(x_{j},\delta^{2\rho})$ can be covered by a single tube of width $\delta^{\rho}$, in any direction desired. So, to prove the lemma, it remains to cover $B_{d}'(X_{j}) \setminus B_{j}$. 

Note that if $x,y$ satisfy $|x - y| \geq \delta^{2\rho}$, then the direction of any $\delta^{4\rho}$-tube containing both $x,y$ lies in a fixed arc $J(x,y) \subset S$ of length $|J(x,y)| \lesssim \delta^{4\rho}/\delta^{2\rho} = \delta^{2\rho}$. As a corollary, the union of all $\delta^{4\rho}$-tubes containing $x,y$, intersected with $B(0,1)$, is contained in a single tube of width $\sim \delta^{2\rho}$. In particular, this union (still intersected with $B(0,1)$) is contained in a single $\delta^{\rho}$-tube, assuming that $\delta > 0$ is small; this tube can be chosen to be a $\delta^{\rho}$-tube around an arbitrary $\delta^{4\rho}$-tube containing both $x$ and $y$.

The tube-cover of $B_{d}'(X_{j}) \setminus B_{j}$ can now be constructed by adding one tube at a time. First, assume that there is a point $y_{1} \in B_{d}'(X_{j}) \setminus B_{j}$ left to be covered, and find a tube $T(\ell_{1},\delta^{4\rho})$ containing both $y_{1}$ and $x_{j}$, with direction in $J_{d}$; existence follows from \eqref{form36}. Add the tube $T(\ell_{1},\delta^{\rho})$ to the the tube-cover of $B_{d}'(X_{j}) \setminus B_{j}$, and recall from the previous paragraph that $T(\ell_{1},\delta^{\rho})$ now contains $T \cap B(0,1)$ for \textbf{any} $\delta^{4\rho}$-tube $T \supset \{y_{1},x_{j}\}$ (of which $T = T(\ell_{1},\delta^{4\rho})$ is just one example). Finally, by definition of $y_{1} \in B_{d}'(X_{j})$, associate to $y_{1}$ a subset $X_{1}' \subset X_{j}$ with 
\begin{equation}\label{form37} \mu(X_{1}') > \delta^{4\eta}/2 \quad \text{and} \quad y_{1} \in B_{d}(X_{1}'). \end{equation}

Assume that the points $y_{1},\ldots,y_{H} \in B_{d}'(X_{j}) \setminus B_{j}$, along with the associated tubes $\{y_{i},x_{j}\} \subset T(\ell_{i},\delta^{4\rho}) \subset T(\ell_{i},\delta^{\rho})$, and subsets $X_{i}' \subset X_{j}$, as in \eqref{form37}, have already been constructed. Assume inductively that 
\begin{equation}\label{form21} \mu(X_{i_{1}}' \cap X_{i_{2}}') \leq \delta^{8\eta}/4, \qquad 1 \leq i_{1} < i_{2} \leq H. \end{equation}
To proceed, pick any point $y_{H + 1} \in B_{d}'(X_{j}) \setminus B_{j}$, and associate to $y_{H + 1}$ a subset $X_{H + 1}' \subset X_{j}$ with $\mu(X_{H + 1}') > \delta^{4\rho}/2$ and $y_{H + 1} \in B_{d}(X_{H + 1}')$. Then, test whether \eqref{form21} still holds, that is, whether $\mu(X_{H + 1}' \cap X_{i}') \leq \delta_{k + 1}^{8\eta}/4$ for all $1 \leq i \leq H$. If such a point $y_{H + 1}$ can be chosen, run the argument from the previous paragraph, first locating a tube $T(\ell_{H + 1},\delta^{4\rho})$ containing both $y_{H + 1}$ and $p_{j}$, with direction in $J_{d}$, and finally adding $T(\ell_{H + 1},\delta^{\rho})$ to the tube-cover under construction. 

The "competing" conditions $\mu(X_{i}') > \delta^{4\eta}/2$, and \eqref{form21}, guarantee that the the algorithm terminates in 
\begin{displaymath} H \leq 4\delta^{-8\eta} \end{displaymath}
steps. The argument is precisely the same as used to prove \eqref{form16}, so I omit it. Once the algorithm has terminated, I claim that all points of $B_{d}'(X_{j}) \setminus B_{j}$ are covered by the tubes $T(\ell_{i},\delta^{\rho})$, with $1 \leq i \leq H$. To see this, pick $y \in B_{d}'(X_{j}) \setminus B_{j}$, and a subset $X' \subset X_{j}$ with $\mu(X') > \delta^{4\eta}/2$, and $y \in B_{d}(X')$. Since the algorithm had already terminated, it must be the case that 
\begin{displaymath} \mu(X' \cap X_{i}') > \delta^{8\eta}/4 \end{displaymath}
for some index $1 \leq i \leq H$. Since $X'' := X' \cap X_{i}' \subset X'$ and consequently $y \in B_{d}(X'')$, one can find a tube $T_{y} = T(\ell_{y},\delta) \ni y$ with $\dir(T_{y}) \in J_{d}$, and satisfying
\begin{displaymath} \mu(X_{i}' \cap T_{y}) \geq \mu(X'' \cap T_{y}) \geq \mu(X'') \cdot \delta^{\tau} > \delta^{8\eta + \tau}/4. \end{displaymath}
This lower bound is precisely the same as in \eqref{form19}. Hence, it follows from the same argument, which gave \eqref{form36}, that
\begin{displaymath} y_{i} \in B_{d}(X_{i}') \subset T(\ell_{y},\delta^{4\rho}). \end{displaymath}
Since $X_{i}' \subset X_{j}$, also $x_{j} \in B_{d}(X_{j}) \subset B_{d}(X_{i}') \subset T(\ell_{q},\delta^{4\rho})$. So, 
\begin{equation}\label{form38} \{y,y_{i},x_{j}\} \subset B(0,1) \cap T(\ell_{y},\delta^{4\rho}). \end{equation}
In particular, $T(\ell_{y},\delta^{4\rho})$ is a $\delta^{4\rho}$-tube containing both $y_{i},x_{j}$, and hence 
\begin{displaymath} B(0,1) \cap T(\ell_{y},\delta^{4\rho}) \subset T(\ell_{i},\delta^{\rho}). \end{displaymath}
Combined with \eqref{form38}, this yields $y \in T(\ell_{i},\delta^{\rho})$, as claimed. This concludes the proof of Lemma \ref{auxLemma3}. \end{proof}

Combining \eqref{form16}-\eqref{form17} with Lemma \ref{auxLemma3}, the structural description of $\textbf{Bad}_{k}^{d}$ is now complete: $\textbf{Bad}_{d}^{k}$ is covered by 
\begin{equation}\label{form39} \leq M \cdot 4\delta^{-8\eta} \leq 8\delta^{-12\eta} \end{equation}
tubes of width $\delta^{\rho}$, with directions in $J_{d}$. For non-adjacent $d_{1},d_{2} \in \{1,\ldots,D\}$ (the ordering of indices corresponds to the ordering of the arcs $J_{d} \subset S$), the covering tubes are then fairly transversal. This is can be used to infer that most point in $E_{k}$ do not lie in many different sets $\textbf{Bad}_{k}^{d}$. Indeed, consider the set $\textbf{BadBad}_{k}$ of those points in $\R^{2}$, which lie in (at least) two sets $\textbf{Bad}_{k}^{d_{1}}$ and $\textbf{Bad}_{k}^{d_{2}}$ with $|d_{2} - d_{1}| > 1$. By Lemma \ref{auxLemma3}, such points lie in the intersection of some pair of tubes $T_{1} = T(\ell_{1},\delta^{\rho})$ and $T_{2} = T(\ell_{2},\delta^{\rho})$ with $\dir(T_{i}) \in J_{d_{i}}$. The angle between these tubes is $\gtrsim \delta^{\eta}$, whence
\begin{displaymath} \diam(T_{1} \cap T_{2}) \lesssim \delta^{\rho - \eta}, \end{displaymath}
and consequently
\begin{equation}\label{form22} \nu(T_{1} \cap T_{2}) \lesssim C_{\nu}\delta^{\kappa_{\nu}(\rho - \eta)} \leq C_{\nu}\delta^{\kappa_{\nu}\rho - 2\eta}. \end{equation}
For $d \in \{1,\ldots,D\}$ fixed, there correspond $\lesssim \delta^{-12\eta}$ tubes in total, as pointed out in \eqref{form39}. So, the number of pairs $T_{1},T_{2}$, as above, is bounded by
\begin{displaymath} \lesssim D^{2} \cdot \delta^{-24\eta} \leq \delta^{-26\eta}. \end{displaymath}
Consequently, by \eqref{form22},
\begin{displaymath} \nu(\textbf{BadBad}_{k}) \lesssim C_{\nu} \delta^{-28\eta + \kappa_{\nu}\rho}. \end{displaymath}
This upper bound is far smaller than $\delta_{k}^{\beta}/2 \leq \nu(E_{k})/2$, taking $0 < \max\{\beta,28\eta\} < \kappa_{\nu}\rho/2$, so that
\begin{equation}\label{form25} 0 < \beta < \kappa_{\nu}\rho - 28\eta. \end{equation}
For such choices of $\beta,\eta$, the next task is then to choose $E_{k + 1} \subset E_{k}$ such that $\nu(E_{k + 1}) \geq \delta_{k + 1}^{\beta}$. Start by writing $G_{k} := E_{k} \setminus \textbf{BadBad}_{k}$, so that
\begin{displaymath} \nu(G_{k}) \geq \nu(E_{k})/2 \geq \delta_{k}^{\beta}/2 \end{displaymath}
by the choice of $\beta$. Now, either 
\begin{equation}\label{form24} \nu\left(G_{k} \cap \textbf{Bad}_{k} \right) \geq \frac{\nu(G_{k})}{2} \quad \text{or} \quad \nu\left(G_{k} \cap \textbf{Bad}_{k} \right) < \frac{\nu(G_{k})}{2}. \end{equation}
The latter case is quick and easy: set $E_{k + 1} := G_{k} \setminus \textbf{Bad}_{k}$ and $K_{k + 1} := K_{k}$. Then $\nu(E_{k + 1}) \geq \nu(E_{k})/4 \geq \delta_{k + 1}^{\beta}$ (assuming that $k \geq k_{0}$ is large enough). Moreover, the set $E_{k + 1}$ no longer contains any points in $\textbf{Bad}_{k}$, so \eqref{form1a} is satisfied at level $k + 1$, by the very definition of $\textbf{Bad}_{k}$, see \eqref{form23}.

So, it remains to treat the first case in \eqref{form24}. Start by recalling from \eqref{form14} that $\textbf{Bad}_{k}$ is covered by the sets $\textbf{Bad}_{k}^{d}$, $1 \leq d \leq D$, so
\begin{displaymath} \nu(G_{k} \cap \textbf{Bad}_{k}^{d}) \geq \frac{\nu(G_{k})}{2D} \geq \frac{\delta^{\eta}\delta_{k}^{\beta}}{4} = \frac{\delta^{\eta + \beta/(1 + \epsilon)}}{4} \end{displaymath}
for some fixed $d \in \{1,\ldots,D\}$. Then, recall from \eqref{form39} that $\textbf{Bad}_{k}^{d}$ can be covered by $\leq 8\delta^{-12\eta}$ tubes of the form $T(\ell,\delta^{\rho})$, with directions in $J_{d}$. It follows that there exists a fixed tube $T_{0} = T(\ell_{0},\delta^{\rho})$ such that
\begin{equation}\label{form40} \dir(T_{0}) \in J_{d} \quad \text{and} \quad \nu(G_{k} \cap T_{0} \cap \textbf{Bad}_{k}^{d}) \geq \frac{\delta^{13\eta + \beta/(1 + \epsilon)}}{32}. \end{equation}
So, to ensure $\nu(G_{k} \cap T_{0} \cap \textbf{Bad}_{k}^{d}) \geq \delta^{\beta}$, choose $\eta > 0$ so small that
\begin{equation}\label{form34} 13\eta + \beta/(1 + \epsilon) < \beta. \end{equation} 
To convince the reader that there is no circular reasoning at play, I gather here all the requirements for $\beta$ and $\eta$ (harvested from \eqref{form33}, \eqref{form25}, and \eqref{form34}):
\begin{displaymath} 0 < \beta < \frac{\kappa_{\nu}\rho}{2} \quad \text{and} \quad 0 < \eta < \min\left\{\frac{\kappa_{\mu}/2 - \tau}{2}, \frac{\kappa_{\nu}\rho}{56},\frac{\epsilon \beta}{13(1 + \epsilon)} \right\} \end{displaymath} 
With such choices of $\beta,\eta$, recalling \eqref{form40}, and assuming that $\delta$ is small enough, the set
\begin{displaymath} E_{k + 1} := G_{k} \cap T_{0} \cap \textbf{Bad}_{k}^{d}. \end{displaymath}
satisfies $\nu(E_{k + 1}) \geq \delta^{\beta}$, which is statement (b) from the lemma. It remains to define $K_{k + 1}$. To this end, recall that $T_{0}$ is a tube around the line $\ell_{0} \subset \R^{2}$. Define
\begin{displaymath} K_{k + 1} := K_{k} \setminus T(\ell_{0},\delta^{\eta/2}). \end{displaymath}
Then, assuming that $\eta/2$ has the form $\eta/2 = (1 + \epsilon)^{-\Gamma - 1}$ for an integer $\Gamma = \Gamma(\epsilon,\kappa_{\mu},\kappa_{\nu},\tau) \in \N$ (this is finally the integer from \eqref{form27}), one has
\begin{equation}\label{form32} \delta^{\eta/2} = \delta_{k - \Gamma}. \end{equation}
Since $T(\ell_{0},\delta_{k - \Gamma}) \cap E_{k - \Gamma} \neq \emptyset$, it follows from the induction hypothesis (i) that
\begin{displaymath} \mu(K_{k} \cap T(\ell_{0},\delta_{k - \Gamma})) \leq (\tfrac{1}{4})^{k - k_{0} + 1}. \end{displaymath} 
Consequently,
\begin{displaymath} \mu(K_{k + 1}) \geq \mu(K_{k}) - (\tfrac{1}{4})^{k - k_{0} + 1} \geq 1 - \sum_{k_{0} \leq j < k + 1} (\tfrac{1}{4})^{j - k_{0} + 1}, \end{displaymath} 
which is the desired lower bound from (a) of the statement of the lemma. So, it remains to verify the non-concentration condition \eqref{form1a} for $E_{k + 1}$ and $K_{k + 1}$. To this end, pick $x \in E_{k + 1}$. First, observe that every tube $T = T(\ell,\delta)$, which contains $x$ and has non-empty intersection with $K_{k + 1} \subset B(0,1) \setminus T(\ell,\delta^{\eta/2})$, forms an angle $\gtrsim \delta^{\eta/2}$ with $T_{0}$. In particular, this angle is far larger than $\delta^{\eta}$. Since $\dir(T_{0}) \in J_{d}$ by \eqref{form40}, this implies that $\dir(T) \in J_{d'}$ for some $|d' - d| > 1$. 

Now, if the non-concentration condition \eqref{form1a} still failed for $x \in E_{k + 1}$, there would exist $N \leq \delta^{-\tau}$ tubes $T(\ell_{1},\delta),\ldots,T(\ell_{N},\delta)$, each containing $x$, and with
\begin{displaymath} \mu\left(K_{k + 1} \cap \bigcup_{i = 1}^{N} T(\ell_{i},\delta) \right) > \delta^{\eta}. \end{displaymath}
By the pigeonhole principle, it follows that the tubes $T(\ell_{i},\delta)$ with $\dir(T_{i}) \in J_{d'}$, for some fixed arc $J_{d'}$, cover a set $X \subset K_{k + 1} \subset K_{k}$ of measure $\mu(X) > \delta^{2\eta}$. This means precisely that $x \in \textbf{Bad}_{k}^{d'}$, and by the observation in the previous paragraph, $|d - d'| > 1$. But $x \in E_{k + 1} \subset \textbf{Bad}_{k}^{d}$ by definition, so this would imply that $x \in \textbf{BadBad}_{k}$, contradicting the fact that $x \in E_{k + 1} \subset G_{k}$. This completes the proof of \eqref{form1a}, and the lemma. \end{proof}

The proof of Theorem \ref{main} is now quite standard:

\begin{proof}[Proof of Theorem \ref{main}] Write $s := \Hd K$, and assume that $s > 0$ and $\Hd E > 0$. Make a counter assumption: $E$ is not contained on a line, but $\Hd \pi_{x}(K) < s/2$ for all $x \in E$. Then, find $t < s/2$, and a positive-dimensional subset $\tilde{E} \subset E$, not contained on any single line, with $\Hd \pi_{x}(K) \leq t$ for all $x \in \tilde{E}$ (if your first attempt at $\tilde{E}$ lies on some line $\ell$, simply add a point $x_{0} \in E \setminus \ell$ to $\tilde{E}$, and replace $t$ by $\max\{t,\Hd \pi_{x_{0}}(K)\} < s/2$). So, now $\tilde{E}$ satisfies the same hypotheses as $E$, but with "$< s/2$" replaced by "$\leq t < s/2$". Thus, without loss of generality, one may assume that 
\begin{equation}\label{form30} \Hd \pi_{x}(K) \leq t < s/2, \qquad x \in E. \end{equation}

Using Frostman's lemma, pick probability measures $\mu,\nu$ with $\spt \mu \subset K$ and $\spt \nu \subset E$, and satisfying the growth bounds \eqref{frostmanBound} with exponents $0 < \kappa_{\mu} < s$ and $\kappa_{\nu} > 0$. Pick, moreover, $\kappa_{\mu}$ so close to $s$ that 
\begin{equation}\label{form28} \kappa_{\mu}/2 > t. \end{equation}
Observe that $\mu(\ell) = 0$ for all lines $\ell \subset \R^{2}$. Indeed, if $\mu(\ell) > 0$ for some line $\ell \subset \R^{2}$, then there exists $x \in E \setminus \ell$ by assumption, and 
\begin{displaymath} \Hd \pi_{x}(K) \geq \Hd \pi_{x}(\spt \mu \cap \ell) \geq \kappa_{\mu} > t, \end{displaymath}
violating \eqref{form30} at once. Finally, by restricting the measures $\mu$ and $\nu$ slightly, one may assume that they have disjoint supports. 

In preparation for using Lemma \ref{mainLemma}, fix $\epsilon > 0$, $0 < \tau < \kappa_{\mu}/2$ in such a way that
\begin{equation}\label{form29} \frac{\tau}{(1 + \epsilon)^{2}} > t. \end{equation}
This is possible by \eqref{form28}. Then, apply Lemma \ref{mainLemmaSimplified} to find the set $K' \subset \spt \mu \subset K$ with 
\begin{displaymath} \mu(K') \geq \frac{1}{2}, \end{displaymath}
the parameters $\eta > 0$ and $k_{0} \in \N$, and the point $x \in E$ satisfying \eqref{form1b}. I claim that 
\begin{equation}\label{form31} \Hd \pi_{x}(K') \geq \frac{\tau}{(1 + \epsilon)^{2}}, \end{equation}
which violates \eqref{form30} by \eqref{form29}. If not, cover $\pi_{x}(K)$ efficiently by arcs $J_{1},J_{2},\ldots$ of lengths restricted to the values $\delta_{k} = 2^{-(1 + \epsilon)^{k}}$, with $k \geq k_{0}$. More precisely: assuming that \eqref{form31} fails, start with an arbitrary efficient cover $\tilde{J}_{1},\tilde{J}_{2},\ldots$ by arcs of length $|\tilde{J}_{i}| \leq \delta_{k_{0}}$, satisfying
\begin{displaymath} \sum_{j \geq 1} |\tilde{J}_{j}|^{\tau/(1 + \epsilon)^{2}} \leq 1. \end{displaymath} 
Then, replace each $\tilde{J}_{j}$ by the shortest concentric arc $J_{j} \supset \tilde{J}_{j}$, whose length is of the form $\delta_{k}$. Note that $\ell(J_{j}) \leq \ell(\tilde{J}_{j})^{1/(1 + \epsilon)}$, so that
\begin{displaymath} \sum_{j \geq 1} |J_{j}|^{\tau/(1 + \epsilon)} \leq \sum_{j \geq 1} |\tilde{J}_{j}|^{\tau/(1 + \epsilon)^{2}} \leq 1. \end{displaymath}
The arcs $J_{1},J_{2},\ldots$ now cover $\pi_{x}(K')$, and there are $\leq \delta_{k}^{-\tau/(1 + \epsilon)}$ arcs of any fixed length $\delta_{k}$. Since $x \notin K'$, for every $k \geq k_{0}$ there exists a collection of tubes $\calT_{k}$ of the form $T(\ell,\delta_{k}) \ni x$, such that $|\calT_{k}| \lesssim \delta_{k}^{-\tau/(1 + \epsilon)}$ (the implicit constant depends on $\dist(x,K')$), and
\begin{displaymath} K' \subset \bigcup_{k \geq k_{0}} \bigcup_{T \in \calT_{k}} T. \end{displaymath}
In particular $|\calT_{k}| \leq \delta_{k}^{-\tau}$, assuming that $\delta_{k}$ is small enough for all $k \geq k_{0}$. Recall that $\mu(K') \geq \tfrac{1}{2}$. Hence, by the pigeonhole principle, one can find $k \in \N$ such that the following holds: there is a subset $K_{k}' \subset K'$ with $\mu(K_{k}') \geq \tfrac{1}{100k^{2}}$ such that $K_{k}'$ is covered by the tubes in $\calT_{k}$. But $1/(100k^{2})$ is far larger than $\delta_{k}^{\eta}$, so this is explicitly ruled out by non-concentration estimate \eqref{form1b}. This contradiction completes the proof. \end{proof}
 
\section{Proof of Theorem \ref{mainRadial}} This section contains the proof of Theorem \ref{mainTechnical}, which evidently implies Theorem \ref{mainRadial}. Fix $\mu \in \calM(\R^{d})$ and $x \in \R^{d} \setminus \spt \mu$. For a suitable constant $c_{d} > 0$ to be determined shortly, consider the weighted measure
\begin{displaymath} \mu_{x} := c_{d}k_{x} \, d\mu, \end{displaymath}
where $k_{x} := |x - y|^{1 - d}$ is the $(d - 1)$-dimensional Riesz kernel, translated by $x$. A main ingredient in the proof of Theorem \ref{mainTechnical} is the following identity:

\begin{lemma}\label{lem1} Let $\mu \in C_{0}(\R^{d})$ (that is, $\mu$ is a continuous function with compact support) and $\nu \in \calM(\R^{d})$. Assume that $\spt \mu \cap \spt \nu = \emptyset$. Then, for $p \in (0,\infty)$,
\begin{displaymath} \int \|\pi_{x\sharp} \mu_{x}\|_{L^{p}(S^{d - 1})}^{p} \, d\nu(x) = \int_{S^{d - 1}} \|\pi_{e\sharp} \mu\|_{L^{p}(\pi_{e\sharp}\nu)}^{p} \, d\calH^{d - 1}(e). \end{displaymath}
Here, and for the rest of the paper, $\pi_{e}$ stands for the orthogonal projection onto $e^{\perp} \in G(d,d - 1)$.
\end{lemma}

\begin{proof} Start by assuming that also $\nu \in C_{0}(\R^{d})$. Fix $x \in \R^{d}$. The first aim is to find an explicit expression for the density $\pi_{x} \mu_{x}$ on $S^{d - 1}$, so fix $f \in C(S^{d - 1})$ and compute as follows, using the definition of the measure $\mu_{x}$, integration in polar coordinates, and choosing the constant $c_{d} > 0$ appropriately:
\begin{align*} \int f(e) \, d[\pi_{x\sharp}\mu_{x}](e) & = \int f(\pi_{x}(y)) \, d\mu_{x}(y) = c_{d} \int \frac{f(\pi_{x}(y))}{|x - y|^{d - 1}} \, d\mu(y)\\
& = \int_{S^{d - 1}} f(e) \int_{\R} \mu(x + re) \, dr \, d\calH^{d - 1}(e)\\
& = \int_{S^{d - 1}} f(e) \cdot \pi_{e\sharp}\mu(\pi_{e}(x)) \, d\calH^{d - 1}(e).  \end{align*} 
Since the equation above holds for all $f \in C(S^{d - 1})$, one infers that
\begin{equation}\label{form1} \pi_{x\sharp}\mu_{x} = [e \mapsto \pi_{e\sharp}\mu(\pi_{e}(x))] \, d\calH^{d - 1}|_{S^{d - 1}}. \end{equation}
Now, one may prove the lemma by a straightforward computation, starting with
\begin{align*} \int \|\pi_{x\sharp} \mu_{x}\|_{L^{p}(S^{d - 1})}^{p} \, d\nu(x) & = \int \int_{S^{d - 1}} [\pi_{x\sharp}\mu_{x}(e)]^{p} \, d\calH^{d - 1}(e) \, d\nu(x)\\
& = \int_{S^{d - 1}} \int_{e^{\perp}} \int_{\pi_{e}^{-1}\{w\}} \left[ \pi_{e\sharp}\mu(\pi_{e}(x)) \right]^{p} \nu(x) \, d\calH^{1}(x) \, d\calH^{d - 1}(w) \, d\calH^{d - 1}(e). \end{align*} 
Note that whenever $x \in \pi_{e}^{-1}\{w\}$, then $\pi_{e}(x) = w$, so the expression $[\ldots]^{p}$ above is independent of $x$. Hence,
\begin{align*} \int \|\pi_{x\sharp} \mu_{x}\|_{L^{p}(S^{d - 1})}^{p} \, d\nu(x) & = \int_{S^{d - 1}} \int_{e^{\perp}} \left[\pi_{e\sharp}\mu(w)\right]^{p} \left( \int_{\pi_{e}^{-1}\{w\}} \nu(x) \, d\calH^{1}(x) \right) d\calH^{d - 1}(w) \, d\calH^{1}(e)\\
& = \int_{S^{d - 1}} \int_{e^{\perp}} \left[\pi_{e\sharp}\mu(w)\right]^{p} \pi_{e\sharp}\nu(w) \, d\calH^{d - 1}(w) \, d\calH^{d - 1}(e)\\
& = \int_{S^{d - 1}} \|\pi_{e\sharp} \mu\|_{L^{p}(\pi_{e\sharp}\nu)}^{p} \, d\calH^{d - 1}(e), \end{align*} 
as claimed.

Finally, if $\nu \in \calM(\R^{d})$ is arbitrary, not necessarily smooth, note that 
\begin{displaymath} x \mapsto \|\pi_{x\sharp}\mu_{x}\|_{L^{p}(S^{d - 1})}^{p} \end{displaymath}
is continuous, assuming that $\mu \in C_{0}(\R^{d})$, as we do (to check the details, it is helpful to infer from \eqref{form1} that $\pi_{x}\mu_{x} \in L^{\infty}(S^{d - 1})$ uniformly in $x$, since the projections $\pi_{e\sharp}\mu$ clearly have bounded density, uniformly in $e \in S^{d - 1}$). Thus, if $(\psi_{n})_{n \in \N}$ is a standard approximate identity on $\R^{d}$, one has
\begin{equation}\label{form2} \int \|\pi_{x\sharp}\mu_{x}\|_{L^{p}(S^{d - 1})}^{p} \, d\nu(x) = \lim_{n \to \infty} \int_{S^{d - 1}} \|\pi_{e\sharp} \mu\|_{L^{p}(\pi_{e\sharp}\nu_{n})}^{p} \, d\calH^{d - 1}(e), \end{equation}
with $\nu_{n} = \nu \ast \psi_{n}$. Since $\pi_{e\sharp}\nu_{n}$ converges weakly to $\pi_{e\sharp}\nu$ for any fixed $e \in S^{d - 1}$, and $\pi_{e\sharp}\mu \in C_{0}(e^{\perp})$, it is easy to see that the right hand side of \eqref{form2} equals
\begin{displaymath} \int_{S^{d - 1}} \|\pi_{e\sharp}\mu\|_{L^{p}(\pi_{e\sharp}\nu)}^{p} \, d\calH^{d - 1}(e). \end{displaymath}
This completes the proof of the lemma. \end{proof}

Here is one more (classical) tool required in the proof of Theorem \ref{mainTechnical}:
\begin{lemma}\label{SobolevLemma} Let $0 < \sigma < d/2$, and let $\mu \in \mathcal{M}(\R^{d})$ be a measure with $\spt \mu \subset B(0,1)$ and $I_{d - 2\sigma}(\mu) < \infty$. Then
\begin{displaymath} \|f\|_{L^{1}(\mu)} \lesssim_{d,\sigma} \sqrt{I_{d - 2\sigma}(\mu)}\|f\|_{H^{\sigma}(\R^{d})} \end{displaymath} 
for all continuous functions $f \in H^{\sigma}(\R^{d})$, where
\begin{displaymath} \|f\|_{H^{\sigma}(\R^{d})} := \left(\int |\widehat{f}(\xi)|^{2}|\xi|^{2\sigma} \, d\xi \right)^{1/2}. \end{displaymath}
\end{lemma}
\begin{proof} See Theorem 17.3 in \cite{Mat}. Since $f$ is assumed continuous here, $|f|$ is pointwise bounded by the maximal function $\widetilde{M}f$ appearing in \cite[Theorem 17.3]{Mat}.  \end{proof}

\begin{proof}[Proof of Theorem \ref{mainTechnical}] Fix $2(d - 1) - s < t < d - 1$. It suffices to prove that if $\nu \in \calM(\R^{d})$ is a fixed measure with $I_{t}(\nu) < \infty$, and $\spt \mu \cap \spt \nu = \emptyset$, then
\begin{displaymath} \pi_{x\sharp}\mu_{x} \in L^{p}(S^{d - 1}) \qquad \text{for } \nu \text{ a.e. } x \in \R^{d}, \end{displaymath}
whenever
\begin{equation}\label{assumptions} 1 < p \leq \min\left\{2 - \tfrac{t}{(d - 1)},\tfrac{t}{2(d - 1) - s}\right\}. \end{equation}
I will treat the numbers $d,p,s,t$ as "fixed" from now on, and in particular the implicit constants in the $\lesssim$ notation may depend on $d,p,s,t$. Note that the right hand side of \eqref{assumptions} lies in $(1,2)$, so this is a non-trivial range of $p$'s. Fix $p$ as in \eqref{assumptions}. The plan is to show that
\begin{equation}\label{form4} \int \|\pi_{x\sharp}\mu_{x}\|_{L^{p}(S^{d - 1})}^{p} \, d\nu(x) \lesssim I_{t}(\nu)^{1/2p}I_{s}(\mu)^{1/2} < \infty. \end{equation}  
This will be done via Lemma \ref{lem1}, but one first needs to reduce to the case $\mu \in C_{0}(\R^{d})$. Let $(\psi_{n})_{n \in \N}$ be a standard approximate identity on $\R^{d}$, and write $\mu_{n} = \mu \ast \psi_{n}$. Then $\pi_{x\sharp}(\mu_{n})_{x}$ converges weakly to $\pi_{x\sharp}\mu_{x}$ for any fixed $x \in \spt \nu \subset \R^{d} \setminus \spt \mu$:
\begin{displaymath} \int f(e) \, d[\pi_{x\sharp}\mu_{x}(e)] = \lim_{n \to \infty} \int f(e) \, d\pi_{x\sharp}(\mu_{n})_{x}(e), \qquad f \in C(S^{d - 1}). \end{displaymath}
It follows that
\begin{displaymath} \|\pi_{x\sharp}\mu_{x}\|_{L^{p}(S^{d - 1})}^{p} \leq \liminf_{n \to \infty} \|\pi_{x\sharp}(\mu_{n})_{x}\|_{L^{p}(S^{d - 1})}^{p}, \quad x \in \spt \nu, \end{displaymath}
and consequently
\begin{displaymath} \int \|\pi_{x\sharp}\mu_{x}\|_{L^{p}(S^{d - 1})}^{p} \, d\nu(x) \leq \liminf_{n \to \infty} \int \|\pi_{x\sharp}(\mu_{n})_{x}\|_{L^{p}(S^{d - 1})}^{p} \, d\nu(x) \end{displaymath}
by Fatou's lemma. Now, it remains to find a uniform upper bound for the terms on the right hand side; the only information about $\mu_{n}$, which we will use, is that $I_{s}(\mu_{n}) \lesssim I_{s}(\mu)$. With this in mind, I simplify notation by denoting $\mu_{n} := \mu$. For the remainder of the proof, one should keep in mind that $\pi_{e\sharp}\mu \in C_{0}^{\infty}(e^{\perp})$ for $e \in S^{d - 1}$, so the integral of $\pi_{e\sharp}\mu$ with respect to various Radon measures on $e^{\perp}$ is well-defined, and the Fourier transform of $\pi_{e\sharp}\mu$ on $e^{\perp}$ (identified with $\R^{d - 1}$) is a rapidly decreasing function. 

We start by appealing to Lemma \ref{lem1}:
\begin{equation}\label{form3} \int \|\pi_{x\sharp}\mu_{x}\|_{L^{p}(S^{d - 1})}^{p} \, d\nu(x) = \int_{S^{d - 1}} \|\pi_{e\sharp} \mu\|_{L^{p}(\pi_{e\sharp}\nu)}^{p} \, d\calH^{d - 1}(e). \end{equation}
The next task is to estimate the $L^{p}(\pi_{e\sharp}\nu)$-norms of $\pi_{e\sharp}\mu$ individually, for $e \in S^{d - 1}$ fixed. I start by recording the standard fact (see for example the proof of Theorem 9.3 in Mattila's book \cite{Mat2}) that $I_{t}(\pi_{e\sharp}\nu) < \infty$ for $\calH^{d - 1}$ almost every $e \in S^{d - 1}$; I will only consider those $e \in S^{d - 1}$ satisfying this condition. Recall that $1 < p \leq t/[2(d - 1) - s]$. Fix $f \in L^{q}(\pi_{e\sharp}\nu)$, with $q = p'$ and $\|f\|_{L^{q}(\pi_{e\sharp}\nu)} = 1$, and note that
\begin{displaymath} I_{2(d - 1) - s}(f \, d\pi_{e\sharp}\nu) = \iint \frac{f(x)f(y) \, d\pi_{e\sharp}\nu(x) \, d\pi_{e\sharp}\nu(y)}{|x - y|^{2(d - 1) - s}} \lesssim I_{t}(\pi_{e\sharp}\nu)^{1/p} \end{displaymath}
by H\"older's inequality. It now follows from Lemma \ref{SobolevLemma} (applied in $e^{\perp} \cong \R^{d - 1}$ with $\sigma = [s - (d - 1)]/2$) that
\begin{align*} \int \pi_{e\sharp}\mu \cdot f \, d\pi_{e\sharp}\nu & \lesssim \sqrt{I_{2(d - 1) - s}(f \, d\pi_{e\sharp}\nu)}\|\pi_{e\sharp}\mu\|_{H^{[s - (d - 1)]/2}}\\
& \lesssim \left(I_{t}(\pi_{e\sharp}\nu) \right)^{1/2p}\left(\int_{e^{\perp}} |\widehat{\pi_{e\sharp}\mu}(\xi)|^{2}|\xi|^{s - (d - 1)} \, d\xi \right)^{1/2}. \end{align*}
Since the function $f \in L^{q}(\pi_{e\sharp}\nu)$ with $\|f\|_{L^{q}(\pi_{e\sharp}\nu)} = 1$ was arbitrary, one may infer by duality that
\begin{displaymath} \|\pi_{e\sharp}\mu\|_{L^{p}(\pi_{e\sharp}\nu)} \lesssim \left(I_{t}(\pi_{e\sharp}\nu)\right)^{1/2p}\left(\int_{e^{\perp}} |\widehat{\pi_{e\sharp}\mu}(\xi)|^{2}|\xi|^{s - (d - 1)} \, d\xi \right)^{1/2}. \end{displaymath}
Now it is time to estimate \eqref{form3}. This uses duality once more, so fix $f \in L^{q}(S^{d - 1})$ with $\|f\|_{L^{q}(S^{d - 1})} = 1$. Then, write
\begin{align*} \int_{S^{d - 1}} & \|\pi_{e\sharp}\mu\|_{L^{p}(\pi_{e\sharp}\nu)} \, \cdot f(e) \, d\calH^{d - 1}(e)\\
& \lesssim \int_{S^{d - 1}} \left(I_{t}(\pi_{e\sharp}\nu)\right)^{1/2p}\left(\int_{e^{\perp}} |\widehat{\pi_{e\sharp}\mu}(\xi)|^{2}|\xi|^{s - (d - 1)} \, d\xi \right)^{1/2} \cdot f(e) \, d\calH^{d - 1}(e)\\
& \lesssim \left( \int_{S^{d - 1}} I_{t}(\pi_{e\sharp}\nu)^{1/p} \cdot f(e)^{2} \, d\calH^{d - 1}(e) \right)^{1/2}\left(\int_{S^{d - 1}} \int_{e^{\perp}} |\widehat{\pi_{e\sharp}\mu}(\xi)|^{2}|\xi|^{s - (d - 1)} \, d\xi \, d\calH^{d - 1}(e) \right)^{1/2}. \end{align*} 
The second factor is bounded by $\lesssim I_{s}(\mu)^{1/2} < \infty$, using (generalised) integration in polar coordinates, see for instance (2.6) in \cite{MO}. To tackle the first factor, say "$I$", write $f^{2} = f \cdot f$ and use H\"older's inequality again:
\begin{displaymath} I \lesssim \left(\int_{S^{d - 1}} I_{t}(\pi_{e\sharp}\nu) \, \cdot f(e)^{p} \, d\calH^{d - 1}(e) \right)^{1/2p} \cdot \|f\|_{L^{q}(S^{d - 1})}^{1/2} \end{displaymath}
The second factor equals $1$. To see that the first factor is also bounded, note that if $B(e,r) \subset S^{d - 1}$ is a ball, then
\begin{displaymath} \int_{B(e,r)} f^{p} \, d\calH^{d - 1} \leq \left(\calH^{d - 1}(B(e,r))\right)^{2 - p} \cdot \left( \int_{S^{d - 1}} f^{q} \, d\calH^{d - 1} \right)^{p - 1} \lesssim r^{(d - 1)(2 - p)}. \end{displaymath}
Thus, $\sigma = f^{p} \, d\calH^{d - 1}$ is a Frostman measure on $S^{d - 1}$ with exponent $(d - 1)(2 - p)$. Now, it is well-known (and first observed by Kaufman \cite{Ka}) that
\begin{displaymath} \int_{S^{d - 1}} I_{t}(\pi_{e\sharp}\nu) \, d\sigma(e) = \iint \int_{S^{d - 1}} \frac{d\sigma(e)}{|\pi_{e}(x) - \pi_{e}(y)|^{t}} \, d\nu(x) \, d\nu(y) \lesssim I_{t}(\nu), \end{displaymath}
as long as $t < (d - 1)(2 - p)$, which is implied by \eqref{assumptions}. Hence $I \lesssim I_{t}(\nu)^{1/2p}$, and finally
\begin{displaymath} \int_{S^{d - 1}} \|\pi_{e\sharp}\mu\|_{L^{p}(\pi_{e\sharp}\nu)} \, \cdot f(e) \, d\calH^{d - 1}(e) \lesssim I_{t}(\nu)^{1/2p}I_{s}(\mu)^{1/2} \end{displaymath}
for all $f \in L^{q}(S^{d - 1})$ with $\|f\|_{L^{q}(S^{d - 1})} = 1$. By duality, it follows that
\begin{displaymath} \eqref{form3} \lesssim I_{t}(\nu)^{1/2p}I_{s}(\mu)^{1/2} < \infty. \end{displaymath}
This proves \eqref{form4}, using \eqref{form3}. The proof of Theorem \ref{mainTechnical} is complete. \end{proof}

\end{document}